\documentclass{article}
\usepackage[centertags]{amsmath}
\usepackage{amsfonts}
\usepackage{amssymb}
\usepackage{amsthm}
\usepackage[all]{xy}
\usepackage{url}
\newtheorem{thm}{Theorem}[section]
\newtheorem{prop}[thm]{Proposition}
\theoremstyle{definition}
\newtheorem{defn}{Definition}[section]
\theoremstyle{definition}

\theoremstyle{definition}

\numberwithin{equation}{section}

\newcommand{\mirr}{\leq^{irr}}
\newcommand{\sirr}{{\scriptscriptstyle irr}}
\newcommand{\free}{\boxplus}
\newcommand{\taup}{\stackrel{\centerdot}{\tau}}
\newcommand{\mup}{\stackrel{\centerdot}{\mu}}
\newcommand{\pip}{\stackrel{\centerdot}{\pi}}
\newcommand{\Max}{\mathrm{Max}}
\newcommand{\Min}{\mathrm{Min}}
\newcommand{\Des}{\mathrm{Des}}
\newcommand{\Exc}{\mathrm{Exc}}
\begin{document}
\author{P. Petrullo and D. Senato}
\date{}
\title{A new explicit formula for Kerov polynomials}
\maketitle
\thispagestyle{empty}
\begin{center}
\textsf{Dipartimento di Matematica e Informatica, Universit\`a
degli Studi della Basilicata, via dell'Ateneo Lucano
10, 85100 Potenza, Italia}.\\
\goodbreak
\small\verb"p.petrullo@gmail.com, domenico.senato@unibas.it"
\end{center}
\begin{abstract}
We prove a formula expressing the Kerov polynomial $\Sigma_k$ as a weighted sum over the lattice of noncrossing partitions of the set $\{1,\ldots,k+1\}$.  In particular, such a formula is related to a partial order $\mirr$ on the Lehner's irreducible noncrossing partitions which can be described in terms of left-to-right minima and maxima, descents and excedances of permutations. This provides a translation of the formula in terms of the Cayley graph of the symmetric group $\frak{S}_k$ and allows us to recover the coefficients of $\Sigma_k$ by means of the posets $P_k$ and $Q_k$ of pattern-avoiding permutations discovered by B\'ona and Simion. We also obtain symmetric functions specializing in the coefficients of $\Sigma_k$.
\end{abstract}
\textsf{\textbf{keywords}:} symmetric group, symmetric functions, Cayley graph, Kerov polynomials, noncrossing partitions.
\section{Introduction}
 The $n$-th free cumulant $R_n$ can be thought as a function $R_n:\lambda\in\mathcal{Y}\rightarrow R_n(\lambda)\in\mathbb{Z}$, defined on the set of all Young diagrams $\cal{Y}$, which we identify with the corresponding integer partition, and taking integer values~\cite{B03}. Indeed, after a suitable representation of a Young diagram $\lambda$ as a function in the plane $\mathbb{R}^2$, it is possible to determine the sequences of integers $x_0,\dots,x_m$ and $y_1,\ldots,y_m$, consisting of the $x$-coordinates of the minima and maxima of $\lambda$, respectively. In this way, if we set
\[\mathcal{H}_{\lambda}(z)=\frac{\prod_{i=0}^m(z-x_i)}{\prod_{i=1}^m(z-y_i)},\]
then $R_n(\lambda)$ is the coefficient of $z^{n-1}$ in the formal Laurent series expansion of $\mathcal{K}_{\lambda}(z)$ such that
\[\mathcal{K}_{\lambda}\bigl( \mathcal{H}_{\lambda}(z)\bigr)=\mathcal{H}_{\lambda}\bigl( \mathcal{K}_{\lambda}(z)\bigr)=z.\]
It can be shown that $R_1(\lambda)=0$ for all $\lambda$. So, the $k$-th Kerov polynomial is a polynomial $\Sigma_k(R_2,\ldots,R_{k+1})$ which satisfies the following identity,
\[\Sigma_k(R_2(\lambda),\ldots,R_{k+1}(\lambda))=(n)_k\frac{\chi^{\lambda}(k,1^{n-k})}{\chi^{\lambda}(1^{n})},\]
where $\chi^{\lambda}(k,1^{n-k})$ denotes the value of the irreducible character of the symmetric group $\frak{S}_n$ indexed by the partition $\lambda$ on $k$-cycles. Two remarkably properties of $\Sigma_k$ have to be stressed. First, it is an \lq\lq universal polynomial\rq\rq, that is it does not depend on $\lambda$ nor on $n$. Second, its coefficients are nonnegative integers. A combinatorial proof of the positivity of $\Sigma_k$ is quite recent and is due to F\'eray~\cite{F08}. Such a proof was then simplified by Do\'lega, F\'eray and \'Sniady~\cite{DFS08}. Until now, several results on Kerov polynomials have been proved and conjectured, see for instance \cite{B07,GR07,La08,St02} and \cite{F09} for a more detailed treatment.\\

Originally, free cumulants arise in the noncommutative context of free probability theory~\cite{NS06}, and their applications in the asymptotic character theory of the symmetric group is due mainly to Biane. In 1992, Speicher~\cite{Sp94} showed that the formulae connecting moments and free cumulants of a noncommutative random variable $X$ obey the M\"obius inversion on the lattice of noncrossing partitions of a finite set. This result highlights the strong analogy between free cumulants and classical cumulants, which are related to the moments of a random variable variable $X$, defined on a classical probability space, via the M\"obius inversion on the lattice of all partitions of a finite set. More recently, Di Nardo, Petrullo and Senato~\cite{DNPS09} have shown how the classical umbral calculus provides an alternative setting for the cumulant families which passes through a generalization of the Abel polynomials.\\

In 1997, it was again Biane~\cite{B97} to show that the lattice $NC_n$ of noncrossing partitions of $\{1,\ldots,n\}$ can be embedded into the Cayley graph of the symmetric group $\frak{S}_n$. So that, it seems reasonable that a not too complicated expression of the Kerov polynomials involving noncrossing partitions, or the Cayley graph of $\frak{S}_n$, would exist. In particular, such a formula, conjectured in \cite{B03}, appeared with a rather implicit description into the papers \cite{DFS08,F08}.\\

In this paper, we state an explicit formula expressing $\Sigma_k$ as a weighted sum over the lattice $NC_{k+1}$. In particular, we introduce a partial order $\mirr$ on the subset $NC_n^{\sirr}$ of $NC_n$ consisting of the noncrossing partitions having $1$ and $k+1$ in the same block. Then, we prove that
\[\Sigma_k=\sum_{\tau\in NC_{k+1}^{\sirr}}\left[\sum_{\pi:\tau\mirr\pi}(-1)^{\ell(\pi)-1}W_{\tau}(\pi)\right]R_{\taup},\]
where $\ell(\pi)$ is the number of blocks of $\pi$, $W_{\tau}(\pi)$ is a suitable weight depending on $\tau$ and $\pi$, and $R_{\taup}=\prod_{B}R_{|B|}$, $B$ ranging over the blocks of $\tau$ having at least $2$ elements.

Since each $\pi\in NC_{k+1}^{\sirr}$ is obtained from a given $\pi'\in NC_k$ simply by inserting $k+1$ in the block containing $1$, then the Biane embedding can be used to translate the formula in terms of the Cayley graph of the symmetric group $\frak{S}_k$.

We also define two slight different versions of the Foata bijection which give rise to a description of $\mirr$ in terms of left-to-right minima and maxima of permutations. Moreover, the maps $\theta$ and $f$,
studied by B\'ona and Simion~\cite{BS}, allow us to compute $\Sigma_k$ via the posets $P_k$ and $Q_k$ of pattern-avoiding permutations ordered by inclusion of descents sets and excedances sets respectively.

Finally, the special structure of the weight $W_{\tau}(\pi)$ makes we able to determine symmetric functions $\mathbf{g}_{\mu}(x_0,\ldots,x_{k-1})$ that specialized in $x_i=i$ return the coefficient of $\prod_{i\geq 2}R_i^{m_i}$ in $\Sigma_k$, for every integer partition $\mu$ of size $k+1$ having $m_i$ parts equal to $i$.\\

\section{Kerov polynomials}
Let $n$ be a positive integer and let
$\lambda=(\lambda_1,\ldots,\lambda_l)$ be an integer partition of
size $n$, that is $1\leq\lambda_1\leq \cdots\leq\lambda_l$ and
$\sum\lambda_i=n$. As is well known, the Young diagram of
$\lambda$ (in the French convention) is an array of $n$
left-aligned boxes, whose $i$-th row consists of $\lambda_i$
boxes. Denote by $\mathcal{Y}_n$ the set of all Young diagram of
size $n$, and set $\mathcal{Y}=\bigcup\mathcal{Y}_n$. From now on,
an integer partition and its Young diagram will be denoted  by the
same symbol $\lambda$.

After a suitable representation of a Young diagram $\lambda$ as a function in the plane $\mathbb{R}^2$~\cite{B03}, it is possible to determine the sequences of integers $x_0,\ldots,x_m$ and $y_1,\ldots,y_m$, consisting of the $x$-coordinates of its minima and maxima respectively. Then, by expanding the rational function
\[\mathcal{H}_{\lambda}(z)=\frac{\prod_{i=0}^m(z-x_i)}{\prod_{i=1}^m(z-y_i)}\]
as a formal power series in $z^{-1}$ one has
\[\mathcal{H}_{\lambda}(z)=z^{-1}+\sum_{n\geq 1}M_n(\lambda)\,z^{-(n+1)}.\]
The integer $M_n(\lambda)$ is said to be the $n$-th \textit{moment} of $\lambda$. Now, define $\mathcal{K}_{\lambda}(z)=\mathcal{H}_{\lambda}^{\scriptscriptstyle<-1>}(z)$, that is $\mathcal{K}_{\lambda}(\mathcal{H}_{\lambda}(z))=\mathcal{H}_{\lambda}(\mathcal{K}_{\lambda}(z))=z$, and consider its expansion as a formal Laurent series,
\[\mathcal{K}_{\lambda}(z)=z^{-1}+\sum_{n\geq 1}R_n(\lambda)\,z^{n-1}.\]
Then, the integer $R_n(\lambda)$ is named the $n$-th \textit{free cumulant} of $\lambda$. It is not difficult to see that $M_1(\lambda)=R_1(\lambda)=0$ for all $\lambda$.

By setting $\mathcal{M}_{\lambda}(z)=z^{-1}\mathcal{H}_{\lambda}(z^{-1})$ and $\mathcal{R}_{\lambda}(z)=z\mathcal{K}_{\lambda}(z)$, we obtain two formal power series in $z$,
\[\mathcal{M}_{\lambda}(z)=1+\sum_{n\geq 1}M_n(\lambda)\,z^{n} \text{ and } \mathcal{R}_{\lambda}(z)=1+\sum_{n\geq 1}R_n(\lambda)\,z^{n},\]
such that
\begin{equation}\label{id:RvsM}
\mathcal{M}_{\lambda}(z)=\mathcal{R}_{\lambda}\left( z\,\mathcal{M}_{\lambda}(z)\right).
\end{equation}

Let $\lambda$ and $\mu$ be two partitions of size $n$, and denote by $\chi^{\lambda}(\mu)$ the value of the irreducible character of $\frak{S}_n$ indexed by $\lambda$ on the permutations of type $\mu$. So that, if $\mu=(k,1^{n-k})$, that is $\mu_1=k$ and $\mu_2=\cdots=\mu_{n-k+1}=1$, then the value of the normalized character $\widehat{\chi}^{\lambda}$ on the $k$-cycles of $\frak{S}_n$ is given
by
\[\widehat{\chi}^{\lambda}(k,1^{n-k})=(n)_k\frac{\chi^{\lambda}(k,1^{n-k})}{\chi^{\lambda}(1^n)},\]

where $(n)_k=n(n-1)\cdots(n+k-1)$. The $k$-th Kerov polynomial is
a polynomial $\Sigma_k$, in $k$ commuting variables, which satisfies the following identity,
\[\Sigma_k(R_2(\lambda),\ldots,R_{k+1}(\lambda))=\widehat{\chi}^{\lambda}(k,1^{n-k}).\]
If we think of $R_n(\lambda)$ as the image of a map $R_n:\lambda\in\mathcal{Y}\rightarrow R_n(\lambda)\in\mathbb{Z}$, then also Kerov polynomials become maps $\Sigma_k=\Sigma_k(R_1,\ldots,R_{k+1})$, which are polynomials in the $R_n$'s, such that $\Sigma_k(\lambda)=\widehat{\chi}^{\lambda}(k,1^{n-k})$.

Since the coefficients of $\Sigma_k$ do not depend on $\lambda$ nor on $n$, but only on $k$, such polynomials are said to be \lq\lq universal\rq\rq. A second remarkably property of Kerov polynomials is that all their coefficients are positive integers. This fact is known as the \lq\lq Kerov conjecture\rq\rq~\cite{K00}. The first proof of the Kerov conjecture was given with combinatorial methods by F\'eray~\cite{F08}. The same author with Do\'lega and \'Sniady~\cite{DFS08} have then simplified the proof.  The following formula for $\Sigma_k$ is due to Stanley~\cite{St02}.
\begin{thm}
Let $\mathcal{R}(z)=1+\sum_{n\geq 2}R_nz^n$. If
\[\mathcal{F}(z)=\frac{z}{\mathcal{R}(z)} \text{ and } \mathcal{G}(z)=\frac{z}{\mathcal{F}^{\scriptscriptstyle <-1>}(z^{-1})},\]
then we have
\begin{equation}\label{id:Kerov1}
\Sigma_k=-\frac{1}{k}\,[z^{-1}]_{\infty}\prod_{j=0}^{k-1}\mathcal{G}(z-j).
\end{equation}
\end{thm}
More precisely, if $[z^n]f(z)$ denotes the coefficient of $z^n$ in the formal power series $f(z)$, then $[z^{-1}]_{\infty}f(z)=[z]f(z^{-1})$. This way, identity \eqref{id:Kerov1} states that $\Sigma_k$ is obtained by
expressing the right-hand side in terms of the free cumulants $R_n$'s.

Moreover, if $\mathcal{M}(z)=1+\sum_{n\geq 1}M_nz^n$, then by virtue of \eqref{id:RvsM} we have $z\,{\mathcal{G}(z)}^{-1}=\mathcal{M}(z^{-1}),$ and \eqref{id:Kerov1} can be rewritten in the following equivalent form,
\begin{equation}\label{id:Kerov2}
\Sigma_k=-\frac{1}{k}\,[z^{k+1}]\prod_{j=0}^{k-1}\frac{1-jz}{\mathcal{M}(\frac{z}{1-jz})}.
\end{equation}
\section{Irreducible noncrossing partitions}
A partition of a finite set $S$ is an unordered sequence $\pi=\{A_1,\ldots,A_l\}$ of its nonempty subsets, such that $A_i\cap A_j=\varnothing$, if $i\neq j$, and $\cup A_i=S$. We say that a partition $\tau$ refines a partition $\pi$, in symbols $\tau\leq \pi$, if and only if each block of $\pi$ is union of blocks of $\tau$. Moreover, if $T\subset S$, the restriction of a partition $\pi$ of $S$ to $T$ is the partition $\pi_{\mid_T}$ obtained by removing from $\pi$ all the elements which do not belong to $T$.\\

There is a beautiful formula, due to Speicher~\cite{Sp94}, related to a special family of set partitions, which gives the expression of the moments $M_n$'s in terms of their respective free cumulants $R_n$'s. Let us recall it.

Denote by $[n]$ the set $\{1,\ldots,n\}$. A partition $\pi=\{A_1,\ldots,A_l\}$ of $[n]$ is said to be a noncrossing partition if and only if  $a,c\in A_i$ and $b,d\in A_j$ implies $i=j$, whenever $1\leq a<b<c<d\leq n$. The set of all the noncrossing partitions of $[n]$ is usually denote by $NC_n$. Its cardinality equals the $n$-th Catalan number $C_n=\frac{1}{n+1}\binom{2n}{n}$. Now, if for all $\pi=\{A_1,\ldots,A_l\}\in NC_n$ we set $R_{\pi}=R_{|A_1|}\cdots R_{|A_l|}$, then the formula of Speicher states that
\[M_n=\sum_{\pi\in NC_n}R_{\pi}.\]

A noncrossing partition $\pi$ of $[n]$ is said to be \textit{irreducible} if and only if $1$ and $n$ lies in the same block of $\pi$. To the best of our knowledge, irreducible noncrossing partitions were introduced by Lehner~\cite{Le02}. According to Lehener's notation, the set of all irreducible noncrossing partitions of $[n]$ will be denoted by $NC_n^{\sirr}$.

By taking the sum of the monomials $R_{\pi}$'s, $\pi$ ranging in $NC_n^{\sirr}$ instead of $NC_n$, one defines a quantity $B_n$ known as a boolean cumulant (see \cite{Le02}),
\begin{equation}\label{id:BvsR}
B_n=\sum_{\pi\in NC_n^{\sirr}}R_{\pi}.
\end{equation}
In particular, if $\mathcal{B}(z)=\sum_{n\geq 1}B_nz^n$, then we have
\begin{equation}\label{id:BvsM}
\mathcal{M}(z)=\frac{1}{1-\mathcal{B}(z)}.
\end{equation}

Note that, a partition of $NC_{n+1}^{\sirr}$ is obtained from a partition of $NC_n$ simply by inserting $n+1$ in the block containing $1$. This fixes a bijection between $NC_n$ and $NC_{n+1}^{irr}$, which proves that $|NC_{n+1}^{irr}|=|NC_n|=C_n.$
If $\mu$ is an integer partition of size $n$, let $\ell(\mu)$ denote the number of its parts $\mu_i$'s, and define $NC_{\mu}^{\sirr}$ to be the subset of $NC_n^{\sirr}$ consisting of all the partitions of \textit{type} $\mu$, namely the partitions $\pi=\{A_1,\ldots,A_l\}$ such that the sequence $(|A_1|,\ldots,|A_l|)$ is a rearrangement of $\mu$. It can be shown that, if exactly $m_i(\mu)$ parts of $\mu$ are equal to $i$, and if $m(\mu)!=m_1(\mu)!\cdots m_n(\mu)!$, then we have
\begin{equation}\label{id:irr_lambda}|NC_{\mu}^{irr}|=\frac{(n-2)_{\ell(\mu)-1}}{m(\mu)!}.\end{equation}
The notion of noncrossing partition can be given for any totally ordered set $S$. In particular, $NC_S^{\sirr}$ will denote the set of all the noncrossing partitions of $S$, such that the minimum and the maximum of $S$ lies in the same block. Let us introduce a partial order on $NC_S^{\sirr}$.
\begin{defn}
Let $\tau,\pi\in NC_S^{\sirr}$. We set $\tau\mirr\pi$ if and only if $\tau\leq\pi$ and the restriction $\tau_{\mid_A}$, of $\tau$ to each block $A$ of $\pi$, is in $NC_A^{\sirr}$. In particular,  we say that $\pi$ \textit{covers} $\tau$ if and only if $\tau\mirr\pi$ and $\pi$ is obtained by joining two blocks of $\tau$.
\end{defn}
For instance, let $\tau=\{\{1,5\},\{2,3\},\{4\}\}$, $\pi=\{\{1,2,3,5\},\{4\}\}$ and $\pi'=\{\{1,5\},\{2,3,4\}\}$. Then $\tau,\pi,\pi'\in NC_5^{irr}$ and $\tau$ refines both $\pi$ and $\pi'$. However, $\tau\mirr\pi$ and in particular $\pi$ covers $\tau$, while it is not true that $\tau\mirr\pi'$, since $\tau_{\mid_{\{2,3,4\}}}=\{\{2,3\},\{4\}\}$ is not irreducible.\\

The singletons (i.e. blocks of type $\{i\}$) of the noncrossing partitions will play a special role. For all $\tau\in NC_n$ we denote by $U(\tau)$ the subset of $[n]$ consisting of all the integers $i$ such that $\{i\}$ is a block of $\tau$, while $\taup$ will be the partition obtained from $\tau$ by removing the singletons. When $\tau,\pi\in NC_n^{\sirr}$ and $\tau\mirr\pi$, then $\pi_{\tau}$ is the restriction of $\pi$ to $U(\tau)$. Note that $\pi_{\tau}\in NC_{U(\tau)}$.\\

We define a tree-representation for the partitions of $NC_n^{\sirr}$ in the following way. Assume $\tau=\{A_1,\ldots,A_l\}\in NC_n^{irr}$ and $\min A_i<\min A_{i+1}$. Construct a labeled rooted tree $t_{\tau}$ by the following steps,
\begin{itemize}
\item choose $A_1$ as the root of $t_{\tau}$,
\item if $2\leq i<j\leq l$ then draw an edge between $A_i$ and $A_j$ if and only if $j$ is the lowest integer such that $\min A_i<\min A_j<\max A_j<\max A_i$,
\item label each edge $\{A_i,A_j\}$ with $\min A_j.$
\end{itemize}
For example, if $\tau=\{\{1,2,7,12\},$ $\{3,5,6\},$ $\{4\},$ $\{8,9\},$ $\{10,11\}\}$ then $t_{\tau}$ is the following tree,
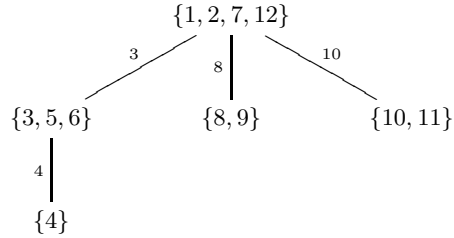
\[\small\xymatrix{
&\{1,2,7,12\}\ar@{-}[dl]_{3}\ar@{-}[d]_{8}\ar@{-}[dr]^{10}&\\
\{3,5,6\}\ar@{-}[d]_{4}&\{8,9\}&\{10,11\}\\
\{4\}&&\\
}\]
Now, let $E(\tau)$ be the set of labels of $t_{\tau}$, and choose $j\in E(\tau)$. We denote by $t_{\tau,j}$ the tree obtained from $t_{\tau}$ by deleting the edge labeled by $j$ and joining its nodes (i.e. joining the blocks). In the following, we will say that $t_{\tau,j}$ is the tree obtained from $t_{\tau}$ by {\lq\lq removing\rq\rq} $j$. Hence, $t_{\tau,3}$ is given by
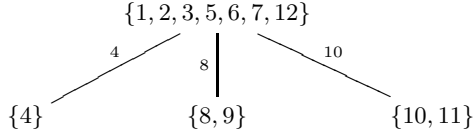
\[\small\xymatrix{
&\{1,2,3,5,6,7,12\}\ar@{-}[dl]_{4}\ar@{-}[d]_{8}\ar@{-}[dr]^{10}&\\
\{4\}&\{8,9\}&\{10,11\}\\
}\]
Of course, $t_{\tau,j}$ is the tree-representation of an irreducible noncrossing partition, here denoted by $\tau_{\{j\}}$, whose blocks are the nodes of $t_{\tau,j}$. By construction, we have $\tau\mirr\tau_{\{j\}}$ and $E(\tau_{\{j\}})=E(\tau)-\{j\}$.
More generally, given a subset $S\subseteq E(\tau)$, we denote by $\tau_S$ the only partition whose tree $t_{\tau_S}$ is obtained from $t_{\tau}$ by removing all labels in $S$ successively. We remark that $\tau_S$ depends only on the set $S$ and not on the order in which labels are choosen. In the example, if $S=\{3,8\}$ then $t_{\tau,S}$ is the tree below,
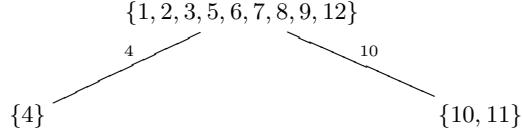
\[\small\xymatrix{
&\{1,2,3,5,6,7,8,9,12\}\ar@{-}[dl]_{4}\ar@{-}[dr]^{10}&\\
\{4\}&&\{10,11\}\\
}\]
This way, we have $\tau_S=\{\{1,2,3,5,6,7,8,9,12\},\{4\},\{10,11\}\}$. The following proposition is easy to prove.
\begin{prop}
Let $\tau,\pi\in NC_n^{\sirr}$. Then, we have $\tau\mirr\pi$ if and only if $\pi=\tau_S$ for some $S\subseteq E(\tau)$. In particular, if $\ell(\tau)$ is the number of blocks of $\tau$, then we have
$$|\{\pi\,|\,\tau\mirr\pi\}|=|2^{E(\tau)}|=2^{\ell(\tau)-1},$$
$2^{E(\tau)}$ denoting the powerset of $E(\tau)$, and
$$|\{\pi\,|\,\pi \text{ covers } \tau\}|=|E(\tau)|=\ell(\tau)-1.$$
\end{prop}
\subsection{The Cayley graph of $\frak{S}_n$}
We start recalling some known results relating noncrossing partitions to the symmetric group in order to describe the partial
order $\mirr$ in terms of permutations.

The Cayley graph of $\frak{S}_n$ is the graph whose nodes are the elements of $\frak{S}_n$ and $w,u\in\frak{S}_n$ are connected by an edge if and only if there exists a transposition $t$ such that $u=wt$.

Denote by $T_n$ the set of all transpositions of $\frak{S}_n$ and for all $w\in\frak{S}_n$ let $\ell_T(w)$ denote the minimum number of transpositions in $T_n$ whose product equals $w$. If we set $u\leq_{T} w$ if and only if $\ell_T(w)=\ell_T(u)+\ell_T(u^{-1}w)$, then we obtain a partial order on $\frak{S}_n$, sometimes called the \textit{absolute order}, whose Hasse diagram can be identified with the Cayley graph.

Biane~\cite{B97} has shown that the lattice $(NC_n,\leq)$ can be embedded into the Cayley graph of $\frak{S}_n$ through a map, here denoted by $\beta$, such that $\tau\leq\pi$ if and only if $\beta(\tau)\leq_T\beta(\pi)$. This embedding has a quite simple description. If $A=\{i_1,\ldots,i_h\}\subseteq [n]$ and $i_1<\cdots<i_h$, let $\beta(A)=(i_1\ldots i_h)\in \frak{S}_n$. Then, set $\beta(\tau)=\beta(A_1)\cdots\beta(A_l)$ whenever $\tau=\{A_1,\ldots,A_l\}$. This way, if $NC(\frak{S}_n)=\{\beta(\tau)\,|\,\tau\in NC_n\}$, then $NC(\frak{S}_n)$ is the interval $[id_n,c_n]=\{w\in\frak{S}_n\,|\,id_n\leq_T w\leq_T c_n\}$ where $id_n=(1)\cdots(n)$ and $c_n=(1\,\ldots\,n)$. Moreover, it is easy to see that if $NC^{\sirr}(\frak{S}_n)=\{\beta(\tau)\,|\,\tau\in NC_n^{\sirr}\}$, then $NC^{\sirr}(\frak{S}_{n+1})=\{(1\,n+1)w\,|\,w\in NC(\frak{S}_n)\}$.\\

Incidentally, the Biane map $\beta$ allows us to obtain a further enumerative result. Indeed, consider the expression of a permutation $w$ as a product of its disjoint cycles, $w=(i_{1,1}\ldots i_{1,n_1})\cdots (i_{l,1}\ldots i_{l,n_l})$. For $1\leq h\leq l$ define
\[T_{(i_{h,1}\ldots i_{h,n_h})}=\{(i\,j)\in T_n\,|\,i_{h,1}\leq i<j< i_{h,n_h}\},\]
and then set
\[T_w=\bigcup_{1\leq h\leq l} T_{(i_{h,1}\ldots i_{h,n_h})}.\]
Now, the following proposition is easy to prove.
\begin{prop}
Let $\tau,\pi\in NC_n^{irr}$, $u=\beta(\tau)$ and $w=\beta(\pi)$. Then, $\pi$ covers $\tau$ if and only if there exists $t\in T_w$ such that $u=wt.$ In particular, we have
\[|\{\tau\in NC_n^{irr}\,|\,\pi \text{ covers } \tau\}|=|T_w|=\sum_{A\in\tau}\binom{|A|-1}{2}.\]
\end{prop}
\subsection{Left-to-right minima and maxima}
Let $w=w_1\ldots w_n$ be a permutation of $\frak{S}_n$ written as a word, that is $w_i=w(i)$. A \textit{left-to-right maximum} of $w$ is an integer $w_i$ such that $w_j<w_i$ for all $j<i$. Analogously, a \textit{left-to-right minimum} of $w$ is an integer $w_i$ such that $w_j>w_i$ for all $j<i$. Of course, every $w\in\frak{S}_n$ has a trivial left-to-right minimum in $1$, and a trivial left-to-right maximum in $n$. Denote by $\Max(w)$ and $\Min(w)$ the sets of all the nontrivial left-to-right maxima and left-to-right minima of $w$, respectively.

There is a well known bijection, named the \textit{Foata bijection}, showing that the number of permutations in $\frak{S}_n$ with $k$ cycles equals the number of permutations in $\frak{S}_n$ with $k$ left-to-right maxima. In this paper we use two slight different versions of the Foata bijection which we are going to describe.\\

Let $w=(i_{1,1}\ldots i_{1,n_1})\cdots (i_{l,1}\ldots i_{l,n_l})\in NC^{\sirr}(\frak{S}_n)$. Arrange the cycles of $w$ in decreasing order of their minima from left to right, and define $\check{w}$ to be the permutation (in the word notation) obtained by removing the parenthesis. Clearly, the minimum of each cycle in $w$ is a left-to-right minimum of $\check{w}$. Moreover, it is easy to see that the map $w\rightarrow\check{w}$ is a bijection.

Now, consider the same permutation $w$. First, arrange the cycles in decreasing order of their maxima from left to right. If $w'$ is the word obtained by removing the parenthesis, then let $\hat{w}$ denote the reflection of $w$ with respect to its middle-point, that is $\hat{w}_i=w'_{n-i+1}$. This way, the maximum of each cycle of $w$ is a left-to-right-maximum of $\hat{w}$, and the map $w\rightarrow\hat{w}$ is a bijection too. Note also that, if $\beta(\tau)=w$, then we have $\Min(\check{w})\cap\Max(\hat{w})=U(\tau)$. In fact, $\Min(\check{w})\cap\Max(\hat{w})$ consists of the fixed points of $w$, that is the singletons of $\tau$.\\

For example, consider $w=(1\,2\,10)$ $(4)$ $(5\,6\,7)$ $(3)$ $(8\,9)\in NC^{\sirr}(\frak{S}_{10})$. By arranging the cycles in decreasing order of their minima we have $w=(8\,9)$ $(5\,6\,7)$ $(4)$ $(3)$ $(1\,2\,10)$, then $\check{w}=8\,9\,5\,6\,7\,4\,3\,1\,2\,10$ and $\Min(\check{w})=\{3,4,5,8\}.$

Therefore, if we arrange the cycles in decreasing order of their maxima we obtain $w=(1\,2\,10)$ $(8\,9)$ $(5\,6\,7)$ $(4)$ $(3)$, so that $w'=1\,2\,10\,8\,9\,5\,6\,7\,4\,3$ and finally $\hat{w}=3\,4\,7\,6\,5\,9\,8\,10\,2\,1$. This way, $\Max(\hat{w})=\{3,4,7,9\}$.
\begin{prop}
Let $\tau,\pi\in NC_n^{irr}$, $w=\beta(\tau)$ and $u=\beta(\pi)$. If $\tau\mirr\pi$ then $\Min(\check{u})\subseteq \Min(\check{w})$ and $\Max(\hat{u})\subseteq \Max(\hat{w})$. Moreover, once fixed $\tau$, the maps $\pi\in\{\pi\,|\,\tau\mirr\pi\}\rightarrow\Min(\check{u})$ and $\pi\in\{\pi\,|\,\tau\mirr\pi\}\rightarrow\Max(\hat{u})$ are bijections.
\end{prop}
\begin{proof}
Note that, the labels of the tree $t_{\tau}$ are exactly the nontrivial minima of the cycles of $w$, that is $E(\tau)=\Min(\check{w})$. Moreover, if we define a new label on $t_{\tau}$ by replacing $\min A_j$ with $\max A_j$, then Proposition~3.1 is again true. So, the proof follows by means of Proposition~3.1.\\
\end{proof}
\subsection{Descents, excedances and pattern-avoiding permutations}
An integer $i\in [n-1]$  is a \textit{descent} for a permutation $w=w_1\ldots w_n\in\frak{S}_n$ if $w_i>w_{i+1}$, while it is called an \textit{excedance} of $w$ if $w_i>i$. We denote by $\Des(w)$  the set of all the descents of $w$, and by $\Exc(w)$ the set of all its excedances.\\

Consider the poset $NC_n$ under the refinement order. Following B\'ona and Simion~\cite{BS}, let $P_n$ denote the set of all $132$-avoiding permutations of $\frak{S}_n$, and let $Q_n$ denote the set of all $321$-avoiding permutations of $\frak{S}_n$. A partial order can be introduced on $P_n$ and $Q_n$ by assuming $u\leq w$ in $P_n$ (resp. in $Q_n$) if and only if $\Des(u)\subseteq\Des(w)$ (resp. $\Exc(u)\subseteq\Exc(w)$).  Then, there are two order-preserving bijections $f:NC_n\rightarrow P_n$ and $\theta:NC_n\rightarrow Q_n$ with the following properties:
\begin{itemize}
\item $i\geq 1$ is a descent of $f(\tau)$ if and only if $i+1$ is the minimum of its block in $\tau$,
\item $i\geq 1$ is an excedance of $\theta(\tau)$ if and only if $i+1$ is the minimum of its block in $\tau$.
\end{itemize}
We also observe that, if $\tau\in\ NC_{n+1}^{\sirr}$ and if $w=f(\tau)$, then $w_{n+1}=n+1$. Analogously, if $w=\theta(\tau)$ then $w_{n+1}=n+1$. Finally, this says that the image of $NC_{n+1}^{\sirr}$ under $f$ (resp. $\theta$) can be identified with $P_n$ (resp. $Q_n$).
\section{Kerov polynomial formula}
By means of the results of Section~2 and Section~3, we are able to give a new formula for the Kerov polynomial $\Sigma_k$. In particular, such a formula is related to the order $\mirr$ on the irreducible noncrossing partitions of the set $[k+1]$. Furthermore, the map $\beta$ makes we able to compute Kerov polynomials via the Cayley graph of $\frak{S}_k$.\\

Let $j$ be a nonnegative integer and denote by $\lambda\free j$ the image of the diagram $\lambda$ under the translation of the plane given by $x\rightarrow x+j$. The $i$-th minimum and maximum of $\lambda\free j$ are $x_i+j$ and $y_i+j$ respectively, so that
\[\mathcal{H}_{\lambda\free j}(z)=\frac{\prod_{i=0}^{m}z-(x_i+j)}{\prod_{i=1}^{m}z-(y_i+j)} \text{ and } \mathcal{M}_{\lambda\free j}(z)=\frac{1}{1-jz}\mathcal{M}_{\lambda}\left(\frac{z}{1-jz}\right).\]
In this way we may rewrite \eqref{id:Kerov2} as follows,
\begin{equation}\label{id:Kerov3}
\Sigma_k(R_2(\lambda),\ldots,R_{k+1}(\lambda))=-\frac{1}{k}\,[z^{k+1}]\prod_{j=0}^{k-1}\frac{1}{\mathcal{M}_{\lambda\free j}(z)}.
\end{equation}
Now, let $R_n(\lambda\free j)$ denote the $n$-th free cumulant of $\lambda\free j$, that is the coefficient of $z^n$ in the formal power series $\mathcal{R}_{\lambda\free j}(z)$ such that $\mathcal{M}_{\lambda\free j}(z)=\mathcal{R}_{\lambda\free j}\left( z\mathcal{M}_{\lambda\free j}(z)\right).$ Hence, it is immediate to verify that $\mathcal{R}_{\lambda\free j}(z)=jz+\mathcal{R}_{\lambda}(z)$, or equivalently
\begin{equation}\label{def:lambda_free_j}
R_n(\lambda\free j)=R_n(\lambda)+j\delta_{1,n},
\end{equation}
where $\delta_{1,n}$ is the Kronecker delta.
\begin{thm}[The formula for Kerov polynomials]
We have
\begin{equation}\label{first formula}\Sigma_{k}=\sum_{\tau\in NC_{k+1}^{irr}}\left[\sum_{\pi\,:\,\tau\mirr\pi}(-1)^{\ell(\pi)-1}W_{\tau}(\pi)\right]R_{\taup},\end{equation}
where
\[W_{\tau}(\pi)=\frac{1}{k!}\sum_{w\in\frak{S}_k}(w(1)-1)^{|A_1|}\cdots (w(k)-1)^{|A_k|},\]
if $\pi_{\tau}=\{A_1,\ldots,A_l\}$ and $A_i=\varnothing$ for $i> l.$
\end{thm}
\begin{proof}
Let $B_n(j)$ denote the $n$-th boolean cumulant $B_n(\lambda\free j)$ of $\lambda\free j$. Since $R_1(\lambda)=0$, then from \eqref{id:BvsR} and \eqref{def:lambda_free_j} we deduce
\begin{equation}\label{id:BvsRj}
B_n(j)=\sum_{\pi\in NC_n^{irr}}j^{u(\pi)}R_{\pip}(\lambda),
\end{equation}
where $u(\pi)=|U(\pi)|.$ Via \eqref{id:BvsM} we have $[z^n]\left(\mathcal{M}_{\lambda\free j}(z)\right)^{-1}=-B_n(j)$, then the right-hand side in (\ref{id:Kerov3}) is equal to
\[\small\sum_{\mu}(-1)^{\ell(\mu)-1}m(\mu)!(k-\ell(\mu))!\sum_{w\in\frak{S}_k}\prod_{i=1}^{k}B_{\mu_i}(w(i)-1).\]
Here $\mu=(\mu_1,\ldots,\mu_l)$ ranges over all the integer partitions of size $k+1$ with at most $k$ parts, and $\mu_i=0$ if $i>\ell(\mu)$. However, by taking into account \eqref{id:irr_lambda} we may rewrite it in the following form,
\[\small\frac{1}{k!}\sum_{\pi}(-1)^{\ell(\pi)-1}\sum_{w\in\frak{S}_k}\prod_{i=1}^{k}B_{|A_i|}(w(i)-1),\]
where $\pi=\{A_1,\ldots,A_l\}$ ranges over all the irreducible noncrossing partitions of $[k+1]$ (which in fact have at most $k$ blocks), and $A_i=\varnothing$ if $i>\ell(\pi)$.\\
The second sum in the expression above equals, via identity (\ref{id:BvsRj}), the following quantity,
\[\small\sum_{\tau_1,\ldots,\tau_k}\sum_{w\in\frak{S}_k}(w(1)-1))^{u(\tau_1)}\cdots (w(k)-1))^{u(\tau_k)}R_{\taup_1}(\lambda)\cdots R_{\taup_k}(\lambda),\]
where $\tau_i$ ranges over all $NC_{A_i}^{irr}$, with $NC_{\varnothing}^{irr}=\varnothing$. Now, if we set $\tau=\tau_1\cup\cdots\cup\tau_k$, then $\tau\in NC_{k+1}^{\sirr}$, $\tau\mirr\pi$ and $R_{\taup}(\lambda)=R_{\taup_1}(\lambda)\cdots R_{\taup_k}(\lambda).$ Finally, $u(\tau_i)$ is the number of singletons in $\tau_i=\tau_{\mid_{A_i}}$, that is the cardinality of the set $A_i\cap U(\tau)$, which if nonempty is a block of $\pi_{\tau}$. This completes the proof.
\end{proof}
So, for all integer partitions $\mu$ of size $k+1$, if $\mup$ is obtained from $\mu$ by removing all parts equal to $1$, then the monomial $R_{\mup}=R_{\mup_1}\cdots R_{\mup_l}$ occurs in $\Sigma_k$ with a nonnegative coefficient. Thanks to \eqref{first formula} and Proposition~3.1 we known that such a coefficient is
\[\sum_{\tau\in NC_{\mu}^{\sirr}}\sum_{\pi:\tau\mirr\pi}(-1)^{\ell(\pi)-1}W_{\tau}(\pi)=\sum_{\tau\in NC_{\mu}^{\sirr}}\sum_{S\subseteq E(\tau)}(-1)^{|E(\tau)|-|S|}\,W_{\tau}(S),\]
where $W_{\tau}(S)=W_{\tau}(\pi)$ if $\pi=\tau_S$.

However, via the map $\beta$ the same coefficient can be recovered on the Cayley graph of $\frak{S}_k$. Indeed, if $u^*=(1\,k+1)u$ for all $u\in NC(\frak{S}_k)$ then we may rewrite it in the following form,
\[\sum_{\scriptscriptstyle u\in NC(\frak{S}_k)\atop u^*\in NC_{\mu}^{\sirr}(\frak{S}_{k+1})}\sum_{w:u\leq_T w}(-1)^{\ell(w)-1}\,W_{u}(w),\]
with $\ell(w)$ and $W_u(w)$ defined in the suitable way.

Proposition~3.3 provides connections between Kerov polynomials and left-to-right minima and maxima. While, the maps $f$ and $\theta$ give relations between $\Sigma_k$ and descents and excedances, and allow us to recover $\Sigma_k$ from the posets $P_k$ and $Q_k$ of B\'ona and Simion.\\

Now, let $\{x_0,\ldots,x_{k-1}\}$ be a set of commuting variables and consider the polynomial $\Omega_k(x_0,\ldots,x_{k-1})$ defined by
\[\Omega_k(x_0,\ldots,x_{k-1})=-\frac{1}{k}[z^{k+1}]\prod_{j=0}^{k-1}\frac{1-x_jz}{\mathcal{M}\bigl(\frac{z}{1-x_jz}\bigr)}.\]
Of course, $\Omega_k$ is symmetric with respect to the $x_i$'s. Moreover, by virtue of \eqref{id:Kerov2} we obtain $\Omega_k(0,1,\ldots,k-1)=\Sigma_k$. A formula for $\Omega_k(x_0,\ldots,x_{k-1})$ is obtained simply by replacing $j$ with $x_j$ in \eqref{first formula}. More precisely, if $\mu$ is an integer partition of size $k+1$, then the coefficient of $R_{\mup}$ in the polynomial $\Omega_k(x_0,\ldots,x_{k-1})$ is given by
\begin{equation}\label{id:omega_coeff}\small
\sum_{\scriptscriptstyle\tau\in NC_{\mu}^{\sirr}}\sum_{S\subseteq E(\tau)}(-1)^{|E(\tau)|-|S|}\,W_{\tau}(S;x_0,\ldots,x_{k-1}),
\end{equation}
where $W_{\tau}(S;x_0,\ldots,x_{k-1})$ is simply obtained by replacing $j$ with $x_j$ in the definition of $W_{\tau}(S)$. Let $\lambda_{\tau}(S)$ denote the integer partition corresponding to the type of $\pi_{\tau}$, with $\pi=\tau_S$. Then, it is not difficult to see that the weight \eqref{id:omega_coeff} satisfies
\[k!\,W_{\tau}(S;x_0,\ldots,x_{k-1})=m(\lambda_{\tau}(S))!\,(k-\ell(\lambda_{\tau}(S)))!\,\mathbf{m}_{\lambda_{\tau}(S)}(x_0,\ldots,x_{k-1}),\]
$\mathbf{m}_{\lambda_{\tau}(S)}(x_0,\ldots,x_{k-1})$ being the monomial symmetric function indexed by $\lambda_{\tau}(S)$ ~\cite{M95}.
So that the coefficient of $R_{\mup}$ in $\Omega_k(x_0,\ldots,x_{k-1})$ is a symmetric function of degree $m_1(\mu)$. Denote it by $\mathbf{g}_{\mu}(x_0,\ldots,x_{k-1})$ and assume
\[\mathbf{g}_{\mu}(x_0,\ldots,x_{k-1})=\sum_{\lambda}g_{\mu,\lambda}\,\mathbf{m}_{\lambda}(x_0,\ldots,x_{k-1}).\]
The left-hand side of \eqref{id:omega_coeff} assures us that, for every $\lambda$ of size $m_1(\mu)$ we have
\[g_{\mu,\lambda}=\frac{1}{k!}\sum_{\tau\in NC_{\mu}^{\sirr}}\sum_{S\subseteq E(\tau)\atop \lambda_{\tau}(S)=\lambda}(-1)^{|E(\tau)|-|S|}\,m(\lambda)!(k-\ell(\lambda))!,\]
hence the $g_{\mu,\lambda}$'s are rational numbers. Moreover,
since $\mathbf{g}_{\mu}(0,\ldots,k-1)$ is the coefficient of
$R_{\mup}$ in $\Sigma_k$, then the Kerov conjecture implies it is
a nonnegative integer. Hence, we may check if the
$g_{\mu,\lambda}$'s are nonnegative integers too. This is not
true. In fact, we have
\[\mathbf{g}_{(3,1,1,1)}=\frac{4}{5}\mathbf{m}_{(1,1,1)}-\frac{3}{5}\mathbf{m}_{(1,2)}+\frac{4}{5}\mathbf{m}_{(3)}.\]
More generally, the expansion of $\mathbf{g}_{(3,1,1,1)}(x_0,\ldots,x_{k-1})$ in terms of  all the classical basis of the ring of symmetric functions, namely elementary functions $\mathbf{e}_{\lambda}$, complete homogeneous functions $\mathbf{h}_{\lambda}$, power sum functions $\mathbf{p}_{\lambda}$ and Schur functions $\mathbf{s}_{\lambda}$, have rational coefficients which are not positive integers. In fact we have,
\begin{align}\mathbf{g}_{(3,1,1,1)}&=\frac{14}{5}\mathbf{h}_{(1,1,1)}-7\mathbf{h}_{(1,2)}+5\mathbf{h}_{(3)}\nonumber\\
                                                       &=\frac{4}{5}\mathbf{e}_{(1,1,1)}-3\mathbf{e}_{(1,2)}+5\mathbf{e}_{(3)}\nonumber\\
                                                       &=\frac{5}{3}\mathbf{p}_{(1,1,1)}-\mathbf{p}_{(1,2)}+\frac{2}{15}\mathbf{p}_{(3)}\nonumber\\
                                                       &=\frac{4}{5}\mathbf{s}_{(1,1,1)}-\frac{7}{5}\mathbf{s}_{(1,2)}+\frac{14}{5}\mathbf{s}_{(3)}.\nonumber
\end{align}

We conclude this paper by stating a second formula expressing $\Sigma_k$ as a weighted sum over the whole $NC_{k+1}$. To this aim, let us introduce the notion of an irreducible component of a noncrossing partition.

Given $\tau\in NC_n$, let $j_1$ be the greatest integer lying in the same block of $1$. Set $\tau_1=\tau_{\mid_{[j_1]}}$ so that $\tau_1$ is an irreducible noncrossing partition of $[j_1]$. Now, let $j_2$ be the greatest integer lying in the same block of $j_1+1$ and set $\tau_2=\tau_{\mid_{[j_1+1,j_2]}}$. By iterating this process, we determine the sequence of irreducible noncrossing partitions $\tau_1,\ldots,\tau_d$, which we name the \textit{irreducible components} of $\tau$, such that $\tau=\tau_1\cup\cdots\cup\tau_d$. For all $\tau\in NC_n$, we denote by $d(\tau)$ the number of its irreducible components. Note that, $d(\tau)=1$ if and only if $\tau$ is an irreducible noncrossing partition. The proof of the following theorem is omitted.
\begin{thm}
We have
\begin{equation}\label{second formula}\Sigma_k=\sum_{\tau\in NC_{k+1}}\left[(-1)^{d(\tau)-1}V_{\tau}\right]R_{\taup},\end{equation}
where
\[V_{\tau}=\frac{1}{k}\sum_{\scriptscriptstyle 0\leq i_1<\ldots<i_{d}\leq k-1}i_1^{u(\tau_1)}\cdots i_{d}^{u(\tau_{d})},\]
if $d=d(\tau).$
\end{thm}

\begin{thebibliography}{99}\small
%
\bibitem{B97} \textsc{P. Biane}, \textit{Some properties of crossings and partitions}, Discrete Math. \textbf{175} (1997), 41-53.
%
\bibitem{B03} \textsc{P. Biane}, \textit{Characters of the symmetric group and free cumulants},
Lecture Notes in Math. \textbf{1815} (2003), Springer,
Berlin, 185-200.
%
\bibitem{B07} \textsc{P. Biane}, \textit{On the formula of Goulden and Rattan for Kerov Polynomials},
S\'{e}m. Lothar. Combin. \textbf{55} (2006).
%
\bibitem{BS} \textsc{M. B\'ona, R. Simion}, \textit{A self-dual poset on objects counted by the Catalan numbers
and a type-B analogue}, Discrete Math. \textbf{220} (2000),  35-49.
%
\bibitem{DFS08} \textsc{M. Do\'lega, V. F\'eray and P. \'Sniady}, \textit{Explicit combinatorial interprettion of Kerov character polynomials as number of permutation factorizations}, arXiv: 0810.3209v2 (2008).
%
\bibitem{DNPS09} \textsc{E. Di Nardo, P. Petrullo and D. Senato}, \textit{Cumulants and convolutions via Abel polynomials}, preprint.
%
\bibitem{F08} \textsc{V. F\'eray}, \textit{Combinatorial interpretation and positivity of Kerov's character polynomials}, J. Algebraic Combin. \textbf{29} (2009), 473-507.
%
\bibitem{F09} \textsc{V. F\'{e}ray}, Ph.D. Thesis (2009), available at \url{http://feray.fr/valentin/soutenance}.
%
\bibitem{GR07} \textsc{I.P. Goulden, A. Rattan}, \textit{An explicit form for Kerov's character polynomials}, Trans. Amer. Math. Soc. \textbf{359} (2007), 3669-3685.
%
\bibitem{K00} \textsc{S.V. Kerov}, talk at IHP Conference (2000).
%
\bibitem{La08} \textsc{M. Lassalle}, \textit{Two positive conjectures for Kerov
polynomials}, Adv. in Appl. Math. \textbf{41} (2008), 407-422.
%
\bibitem{Le02} \textsc{F. Lehner}, \textit{Free cumulants and enumeration of
connected partitions}, Europ. J. Combin. \textbf{22} (2002),
1025-1031.
%
\bibitem{M95} \textsc{I.G. Macdonald}, \textit{Symmetric functions and Hall
polynomials}, second ed., Oxford University Press, Oxford (1995).

\bibitem{NS06} \textsc{A. Nica and R. Speicher}, \textit{Lectures on the Combinatorics of Free Probability}, Cambridge University Press (2006).
%
\bibitem{Sp94} \textsc{R. Speicher}, \textit{Multiplicative functions on the lattice on nocrossing partitions and free convolution}, Math. Ann., \textbf{298} (1994), 611-628.
%
\bibitem{St02} \textsc{R.P. Stanley}, \textit{Kerov's character polynomial and irreducible symmetric group charcters of rectangular shape},
Tranparencies from a talk at CMS meeting (2002), Quebec City.
%
%
\end{thebibliography}
\end{document}